\documentclass{article}

\usepackage{amssymb}
\usepackage{amsmath}
\usepackage{amsthm}
\usepackage[utf8]{inputenc}
\usepackage{graphicx}
\usepackage{xcolor,url}

\newtheorem{theorem}{Theorem}
\newtheorem{corollary}{Corollary}

\newtheorem{lemma}{Lemma}
\theoremstyle{definition}
\newtheorem{remark}{Remark}

\title{Counting locally flat-foldable origami configurations via
3-coloring graphs}
\author{Alvin Chiu\thanks{Georgia Institute of Technology, Atlanta, GA, \protect\url{achiu32@gatech.edu}} 
\and
 William Hoganson\thanks{Swarthmore College, Swarthmore, PA, \protect\url{whogans1@swarthmore.edu}}
 \and
 Thomas C. Hull\thanks{Western New England University, Springfield, MA, \protect\url{thull@wne.edu}}
 \and
 Sylvia Wu\thanks{Clemson University, Clemson, SC, \protect\url{sylviaw@g.clemson.edu}}}
\date{}

\begin{document}

\maketitle

\begin{abstract}
Origami, where two-dimensional sheets are folded into complex structures, is proving to be rich with combinatorial and geometric structure, most of which remains to be fully understood. In this paper we consider \emph{flat origami}, where the sheet of material is folded into a two-dimensional object, and consider the mountain (convex) and valley (concave) creases made by such foldings, called a \emph{MV assignment} of the crease pattern.  
We establish a method to, given a flat-foldable crease pattern $C$ under certain conditions, create a planar graph $C^*$ whose 3-colorings are in one-to-one correspondence with the locally-valid MV assignments of $C$.  This reduces the general, unsolved problem of enumerating locally-valid MV assignments to the enumeration of 3-colorings of graphs.
\end{abstract}

\section{Introduction}

Origami is the art of folding paper, allowing the transformation of flat pieces of paper into two- or three-dimensional shapes. Enumerating the different ways to fold a flat sheet along straight crease lines into a flat object is an intriguing, open problem with applications in math, physics, and engineering \cite{Assis,Evans,DiF}. 

Whenever paper is folded, the resulting line of the fold is a $\textit{crease}$. An $\textit{origami crease pattern } (C, P)$ is a straight-line embedding of a planar graph $C = (V(C), E(C))$ on a region $P$ of $\mathbb{R}^2$, where the edges of $C$ 
are the creases of the folded paper. 
Here we assume that $P$ is bounded and the graph $C$ is finite.
Unless the shape of $P$ is important we will refer to a crease pattern as just $C$. Giving the paper an arbitrary ``top" and ``bottom" when folding along the creases, each crease will either be a $\textit{mountain crease}$ (it bends the paper in a convex direction) or $\textit{valley crease}$ (it bends the paper in a concave direction). Thus, the folded state of the paper is a \textit{mountain-valley (MV) assignment} on $C$, which is a function $\mu: E(C) \to \{-1,1\}$, where $\mu(c) = 1$ if the crease $c$ is a mountain crease and $\mu(c) = -1$ if it is a valley crease.

We call a MV assignment $\mu$ \textit{valid} if it can be used to physically fold $(C,P)$ flat such that the model could be pressed between the pages of a book with no self intersections or new creases being made. A crease pattern $C$ is \textit{(globally) flat-foldable} if there exists a valid MV assignment $\mu$ on $C$. This notion of global flat foldability tells us if it is possible to fold the paper flat using all of the creases in the crease pattern. Enumerating all globally valid MV assignments $\mu$ on a given flat-foldable origami crease pattern $C$ with multiple vertices is an open problem \cite{miura}. Just determining the global flat foldability of a specific MV assignment is a difficult problem since the size of the faces in a crease pattern dictates whether or not the layers of the paper will collide. In fact, global flat-folability has been proven to be NP-hard \cite{Bern}, even for simplified crease patterns \cite{boxpleat}. We will focus specifically on MV assignments that are \textit{locally-valid}, meaning every vertex in the set $V(C)$ is valid individually. We call a crease pattern $C$ \textit{locally flat-foldable} if there exists a locally-valid MV assignment $\mu$ on $C$. 

By definition, a single-vertex crease pattern is globally flat-foldable if and only if it is locally flat-foldable. In fact, there is a linear-time algorithm for determining the number of valid MV assignments given any single-vertex crease pattern \cite{countingMV}.  However, computing the number of locally-valid MV assignments for general crease patterns with many vertices is open. A significant advance in this area was given by Ginepro and Hull in \cite{miura}, where the number of locally-valid MV assignments of the $m\times n$ Miura-ori crease pattern (see Figure~\ref{fig:Miura}) was shown to be equal to the number of proper 3-vertex colorings of an $m\times n$ grid graph with one vertex pre-colored.  

In this paper, we expand this result to a wide class of flat-foldable crease patterns.  Specifically, if $C$ is a flat-foldable crease pattern where each vertex satisfies a recursive constraint we call \emph{3-nice} (defined in Section~\ref{sec:backgrounnd}), then we can find a corresponding graph $C^*$ such that the number of locally-valid MV assignments of $C$ equals the number of proper 3-colorings, with one vertex pre-colored, of $C^*$. As a consequence, we prove that several families of crease patterns have the same number of locally-valid MV assignments as the $m\times n$ Miura-ori. Further, the 3-nice property includes all degree-4 flat-foldable vertices, implying that this color correspondence works for all 4-regular crease patterns.
This allows us to reduce our enumeration problem into the more extensively-studied problem of enumerating graph colorings.  It also provides very strong evidence that the combinatorial structure underlying locally-valid MV assignments is 3-colorings of graphs.

\section{Background and notation}\label{sec:backgrounnd}

For background on flat-foldable crease patterns see \cite{boxpleat,miura,countingMV,Hull2}.  We summarize here three main conditions that a single-vertex crease pattern must satisfy in order to fold flat as well as a recursive algorithm for folding such vertices.

\begin{theorem}
\textup{(Kawasaki \cite{countingMV})} Let $(C, P)$ be an origami crease pattern where $C$ has only one vertex $v$ in the interior of $P$ and all edges in $C$ are adjacent to $v$. Let $\alpha_1, \dotsc , \alpha_k$ be the sector angles, in order, between the consecutive edges around v. Then the crease pattern (C,P) is flat-foldable if and only if k = 2n is even and
$$\alpha_1 - \alpha_2 + \alpha_3 - \cdots - \alpha_{2n} = 0.$$
\end{theorem}

In this paper, we only investigate crease patterns which can fold flat, making satisfying Kawasaki's Theorem a precondition for all crease patterns considered.

The next theorems describe conditions that must be satisfied for a particular single-vertex MV assignment to fold flat. These are central to the goal of enumerating all locally-valid MV assignments.

\begin{theorem}
\textup{(Maekawa \cite{countingMV})} Let $v$ be a vertex in a flat-foldable crease pattern $C$ with a valid MV assignment $\mu$ and let $E(v)$ be the set of crease edges adjacent to $v$. Then 
$$\sum_{c \in E(v)} \mu(c) = \pm 2.$$
\end{theorem}

In flat origami, the \textit{Big-Little-Big Theorem} says that if a flat-foldable vertex has a sector angle $\alpha_i$ that is strictly smaller than its two neighboring angles $\alpha_{i-1}$ and $\alpha_{i+1}$, then the creases surrounding $\alpha_i$ cannot be both mountains or both valleys in any valid MV assignment (otherwise the two large angles would totally cover $\alpha_i$ on the same side of the paper, causing a self-intersection).  We will need the following generalization of this given in \cite{countingMV,Hull2}. 

\begin{theorem}[Big-Little-Big]\label{thm:BLB} 
Let $v$ be a vertex in a flat-foldable crease pattern $C$ with a MV assignment $\mu$. Suppose that we have a local minimum of consecutive equal sector angles $\alpha_i$ between the crease edges $c_i, \dotsc, c_{i+k+1}$ at $v$. That is, $\alpha_i = \alpha_{i+1} = \cdots = \alpha_{i+k}$ where $\alpha_{i-1} > \alpha_i$ and $\alpha_{i+k+1} > \alpha_i$. Then $\mu$ is valid among the creases bordering the angles $\alpha_{i},\cdots, \alpha_{i+k}$ if and only if
$$\sum_{j=i}^{i+k+1} \mu(c_i) = \left\{
	\begin{array}{ll}
		0  & \textrm{if k is even,} \\
		\pm 1 & \textrm{if k is odd.}
	\end{array}
\right.$$
\end{theorem}
    Therefore, if there are an even number of small equal angles in a row, the number of mountains and number of valleys bordering those small equal angles must differ by 1. If there is an odd number of small equal angles in a row, there must be the same number of mountains and valleys.

Using Theorem~\ref{thm:BLB}, one can create a recursive algorithm to count the number of valid MV assignments of a flat-foldable single-vertex crease pattern \cite{countingMV}.  Since the nature of this recursion is important to our work, we provide a few details of it here.  Let $C_0$ be our single-vertex, flat-foldable crease pattern, and identify a local minimum of consecutive equal sector angles $\alpha_i,\ldots, \alpha_{i+k}$ as per the Big-Little-Big Theorem.  Fold the creases surrounding these angles and fuse the layers of paper together to get a new crease pattern $C_1$ with fewer creases.  (Note that $C_1$ will now be a crease pattern on a cone-shaped piece of paper, but as detailed in $\cite{countingMV,Hull2}$ these single-vertex flat foldability results, like the Big-Little-Big Theorem, apply to crease patterns on cones as well as flat paper, so the recursion can proceed.)  We then repeat, finding a local minimum of consecutive equal sector angles of $C_1$ and fold them to obtain crease pattern $C_2$.  We proceed in this way until a crease pattern $C_z$ is obtained where the sector angles are all equal, in which case any MV assignment that satisfies Maekawa's Theorem will suffice.  Each recursive step allows us to count how many MV assignments will satisfy the application of the Big-Little-Big Theorem, and in this way we can count the number of valid MV assignments in linear time.  

We define a single-vertex, flat-foldable crease pattern $C$ to be \emph{$m$-nice} if every application of the Big-Little-Big Theorem in its folding recursion with $\alpha_i=\cdots=\alpha_{i+k}$, $\alpha_{i-1}>\alpha_i$, and $\alpha_{i+k+1}>\alpha_i$, we have $k\leq m$.  Note that the definition of $m$-nice does not apply to vertices where all the sector angles are equal.  As we will see in Section~\ref{sec:colorbij}, requiring the vertices in our flat-foldable crease patterns to be 3-nice will be necessary for some of our work.  


Together, these results give us a set of conditions by which local flat foldability is dictated. They allow us to understand and find all valid MV assignments for any single-vertex crease pattern. Our goal now is to find the number of all locally-valid MV assignments for any given crease pattern $C$. If we let $M(C)$ be the set of all locally-valid MV assignments $\mu$ on $C$, our goal is to find $\lvert M(C)\rvert$. 

\begin{figure}
    \centering
    \includegraphics[width=\linewidth]{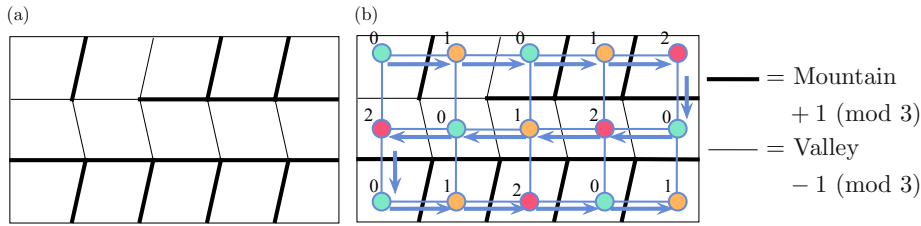}
    \caption{(a) The Miura-ori crease pattern, with a locally-valid MV assignment.  (b) A grid graph superimposed with the bijection scheme illustrated.}
    \label{fig:Miura}
\end{figure}

To do so, we build off of the work of Ginepro and Hull \cite{miura}, which looked at local flat foldability for the Miura-ori crease pattern. In particular, they found a bijection between $M(C)$, where $C$ is a Miura-ori crease pattern made of an $m\times n$ array of parallelograms and $S(C^*)$,  the set of proper $3$-colorings of an $m\times n$ grid graph $C^*$ with one vertex pre-colored. If we overlay $C^*$ on top of $C$ then we can demonstrate the bijection by following a zig-zag path through $C^*$, as shown in Figure~\ref{fig:Miura}(a).  The colors used are the elements of $\mathbb{Z}_3$, and if we pre-color the upper-left vertex with 0, then as we follow the zig-zag path we add 1 (mod 3) to get the next color if we cross a mountain crease and subtract 1 if we cross a valley crease.  For a proof that this is in fact a bijection, see \cite{miura}.

In generalizing this idea, we first establish some notation. Let a proper $3$-coloring of a graph $G = (V, E)$ with one vertex pre-colored be $s: V \to \mathbb{Z}_3$, where $s(u) \neq s(v)$ whenever $\{u, v\}$ is an edge and, arbitrarily, $s(v_0) = 0$. We denote the set of all such proper $3$-colorings $s$ of a graph $G$ as $S(G)$.  

To count locally-valid MV assignments of a flat-foldable crease pattern $(C,P)$ we want to find a graph $C^*$, which we will call the \emph{SAW graph of $C$}, such that there is a bijection $f: M(C) \to S(C^*)$ that maps locally-valid MV assignments of $C$ to proper $3$-colorings with one vertex pre-colored of graph $C^*$. 
To create a (useful) bijection, we want to embed our desired graph $C^*$ onto the paper $P$ and direct some (or perhaps all) of the  edges in $C^*$ satisfying the following conditions:
\begin{itemize}
    \item Each face in $C$ has at least one vertex $v \in V(C^*)$.
    \item For all creases $c_i\in E(C)$ bordering faces $F_i$ and $F_{i+1}$ in $(C,P)$, there must exist two vertices $v_i, v_{i+1}\in V(C^*)$ embedded in $P$ so that $v_i\in F_i$, $v_{i+1}\in F_{i+1}$, and either $(v_i, v_{i+1})$ or $(v_{i+1}, v_i)$ is a directed edge crossing crease $c_i$. 
    \item If $(v_i, v_j)$ is a directed edge in $E(C^*)$ and crosses over a crease $c_i \in E(C)$, then
    $$\mu(c_i) = \left\{
	\begin{array}{ll}
		1  & \textrm{if } s(v_j) - s(v_i) \equiv 1 \pmod{3} \\
		-1 & \textrm{if } s(v_j) - s(v_i) \equiv 2 \pmod{3}
	\end{array}
\right.$$
\end{itemize}
The conditions governing directed edges yield a mapping from any proper $3$-coloring to a (hopefully valid) MV assignment and vice versa. 

For example, the $m \times n$ grid graph in Figure~\ref{fig:Miura} serves as the SAW graph of the $m\times n$ Miura-ori tessellation, as proved in \cite{miura}.
 
Note that given a locally-valid MV assignment of $C$, the directed edges in $C^{*}$ that assign colors based on the assignment of creases do not necessarily color all vertices in the SAW graph $C^{*}$. In order for there to be a bijection from $M(C)$ to $S(C^*)$, the vertices colored by the directed  edges must force one and only one proper coloring on the remaining vertices. 

In the case where $C$ is a single-vertex crease pattern, we will impose an additional condition that in the embedding of $C^*$ on the paper $P$, all the directed edges of $C^*$ will border the outside face of $C^*$. For this reason, we will refer to a directed edge in a single-vertex SAW graph $C^{*}$ as a \textit{boundary edge}. (This will be useful in Section~\ref{sec:tiling} when we tile single-vertex SAW graphs to make SAW graphs for larger crease patterns.)  Therefore, when we color only the vertices on boundary edges of a single-vertex SAW graph $C^*$, this has to determine a unique proper $3$-coloring of $C^{*}$ by forcing a coloring on any interior vertices.
We proceed by finding SAW graphs for single-vertex crease patterns.

\section{Coloring bijections for single-vertex folds}\label{sec:colorbij}

\begin{figure}
    \centering
    \includegraphics[width=\linewidth]{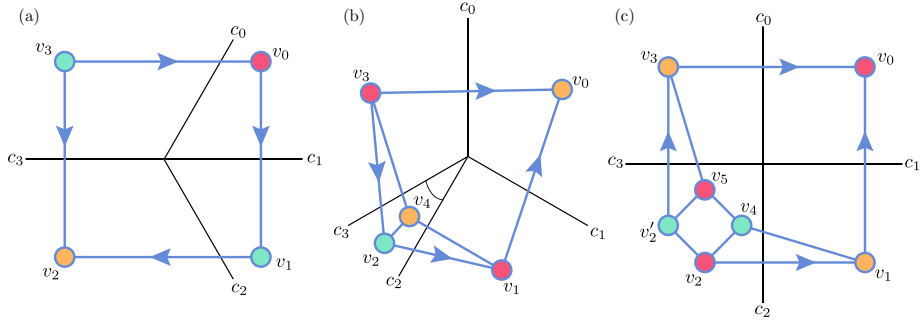}
    \caption{(a) SAW graph for a degree-4 Miura-ori vertex (the so-called bird's foot).  (b) SAW graph for a degree-4, big-little-big vertex. (c) SAW graph for a degree-4 vertex with all $90^\circ$ angles.}
    \label{fig:deg4}
\end{figure}


Degree-4 single vertex crease patterns come in three types, shown in Figure~\ref{fig:deg4}.  The so-called bird's foot (Figure~\ref{fig:deg4}(a)) is made of two congruent acute and two congruent obtuse angles; it is the vertex used in the Miura-ori crease pattern.  If the vertex has one smallest sector angle, then we call it a Big-Little-Big (or BLB for short) vertex (Figure~\ref{fig:deg4}(b)).  The last case is where all the sector angles equal $90^\circ$ (Figure~\ref{fig:deg4}(c)).  As shown in \cite{countingMV,Hull2}, degree-4 BLB vertices have 4 valid MV assignments, bird's feet have 6, and the all-equal-angles vertex has 8.  That is, $|M(C)|=4$, 6, and 8 for these vertices.

\begin{theorem}\label{thm:deg4}
The graphs shown in Figure~\ref{fig:baseCaseEx} are SAW graphs for the three types of degree-4 single vertex crease patterns.
\end{theorem}

\begin{proof}
That Figure~\ref{fig:deg4}(a) is a SAW graph for the bird's foot vertex was proven in \cite{miura}, but we describe the proof here for completeness.  The SAW graph $C^*$ in this case is a 4-cycle, which has 18 ways to properly 3-color the vertices, and thus 6 ways to 3-color with one vertex pre-colored, i.e., $|S(C^*)|=6$.  Therefore to show that the function $f: M(C)\to S(C^*)$ given by the directed edges shown in Figure~\ref{fig:deg4}(a) is a bijection, we only need to establish surjectivity.  The Big-Little-Big Theorem, together with Maekawa's Theorem, tells us that in order for a bird's foot MV assignment $\mu$ to be valid, it needs have $\mu(c_3)$ (the ``heel" of the bird's foot) equal the majority assignment of the ``toes", which is $\mu(c_0)+\mu(c_1)+\mu(c_2)$.  If we pick an arbitrary $s\in S(C^*)$, then let $\mu(c_i)=s(v_i)-s(v_{i-1})$ for $i=0, 1, 2$, where the indices are taken mod 3, and let $\mu(c_3)=s(v_2)-s(v_3)$.  Then
\begin{eqnarray*}
\mu(c_0)+\mu(c_1)+\mu(c_2) & = & s(v_0) - s(v_3) + s(v_1) - s(v_0) + s(v_2) - s(v_1)\\ 
& = & s(v_2)-s(v_3) = \mu(c_3).
\end{eqnarray*}
Therefore we have found a valid MV assignment $\mu$ with $f(\mu)=s$, as desired.

For the BLB crease pattern $C$ and graph $C^*$ shown in Figure~\ref{fig:deg4}(b), we have that $|M(C)|=4$ because we must have $\mu(c_2)\not=\mu(c_3)$ by Big-Little-Big (giving us $(\mu(c_2),\mu(c_3))$ is either $(-1,1)$ or $(1,-1)$) and thus $\mu(c_0)$ and $\mu(c_1)$ must be the same in order to satisfy Maekawa's Theorem (giving us two choices for $\mu(c_0)=\mu(c_1)$).  We also have $|S(C^*)|=4$; if we pre-color $s(v_1)=0$, then the two triangles force $s(v_3)=0$, and then we have two coloring choices for the vertex set $\{v_2, v_4\}$ and two choices for $v_0$.  To show surjectivity of the map $f:M(C)\to S(C^*)$ given by the directed edges shown, we may assume $s(v_1)=s(v_3)=0$ for all $s\in S(C^*)$ (making $v_1$ our pre-colored vertex).  Then the color $s(v_2)$ will determine whether $(\mu(c_2), \mu(c_3))=(-1,1)$ or $(1,-1)$.  Then the color $s(v_0)$ will make either $\mu(c_0)=\mu(c_1)$ be $-1$ or $1$.  All options result in a valid MV assignment $\mu$.

For the all-equal-angles degree-4 case of Figure~\ref{fig:deg4}(c), we know that $|M(C)|=8$ and it is not hard to check that $|S(C^*)|=8$ as well.  Now let $s\in S(C^*)$ with $s(v_0)=0$.  We will show that the MV assignment $\mu$ given by the directed edges in Figure~\ref{fig:deg4}(c) will satisfy Maekawa's Theorem and thus be valid.  Suppose that $s(v_0)=s(v_2)$.  Then $\mu(c_1)\not= \mu(c_2)$.  Now, if $s(v_2)=1$ then $s(v_2)=2$ is forced by $s(v_0)=0$, and we have $\mu(c_0)=\mu(c_3)=1$.  Otherwise $s(v_2)=2$ and $s(v_3)=1$, which means $\mu(c_0)=\mu(c_3)=-1$.  Both cases satisfy Maekawa's Theorem.  On the other hand, if $s(v_0)\not= s(v_2)$ then we must have $s(v_4)=s(v_0)$ and $\mu(c_1)=\mu(c_2)$.  Then, if $s(v_5)=1$ we have $s(v_3)=2$ and $s(v_2)=0$, making $\mu(c_0)\not=\mu(c_3)$, whereas if $s(v_5)=2$ we'll also get $\mu(c_0)\not=\mu(c_3)$.  This covers all cases of colorings $s\in S(C^*)$, and we have that $C^*$ is a SAW graph for $C$.

\end{proof}

\begin{remark}\label{remark1}
Note that when it comes to counting MV assignments, we only care about the number of 3-colorings (with one vertex pre-colored) of the SAW graph.  The directed edges only describe how to biject a given MV assignment to a specific 3-coloring.  

In fact, the directed edges shown in Figure~\ref{fig:deg4} are not the only ones that can be used to perform the bijections for degree-4 flat-foldable vertices.  For the bird's foot, we may reverse all the directed edges and the bijection proof will still work.  For the BLB degree-4 vertex, we may reverse the directed edges along the 3-cycles (edges $(v_3,v_2)$ and $(v_2,v_1)$) \emph{or} we may reverse the other pair of directed edges ($(v_1,v_0)$ and $(v_3,v_0)$) and the proof will still work.  For the all-equal-angles degree-4 vertex, there are several variations of  directed edges that will work.  I.e., if we only switch the directed edges $(v_1,v_0)$ and $(v_3,v_0)$, or if we only switch the edges $(v_1,v_0)$ and $(v_2,v_3)$, then one may check that our bijection proof can still be made to work.  These alternate directed edge assignments are useful for tiling these SAW graphs in larger degree-4 crease patterns, as we will see in Section~\ref{sec:tiling}.
\end{remark}


It is more complicated to find SAW graphs for higher-degree flat-foldable vertices, mainly because of the many different ways the Big-Little-Big Theorem can be applied recursively when the vertex degree becomes large.  Nonetheless, we can design SAW graphs for flat-foldable vertices that are 3-nice.  To see this, we begin with a Lemma.


\begin{lemma}\textup{(Baby SAW graphs.)}\label{lem:babes}
Let $G_k$ be the creases $c_1,\ldots, c_{k+1}$ surrounding a local minimum of $k$ consecutive equal sector angles $\alpha_1,\ldots, \alpha_k$ in a single-verex, flat-foldable crease pattern.  Then for $k=1, 2, 3$ the graphs $G_k^*$ shown in Figure~\ref{fig:babes} are SAW graphs for the creases $G_k$, which we refer to as \textup{baby SAW graphs}.
\end{lemma}

\begin{figure}
    \centering
    \includegraphics[width=\linewidth]{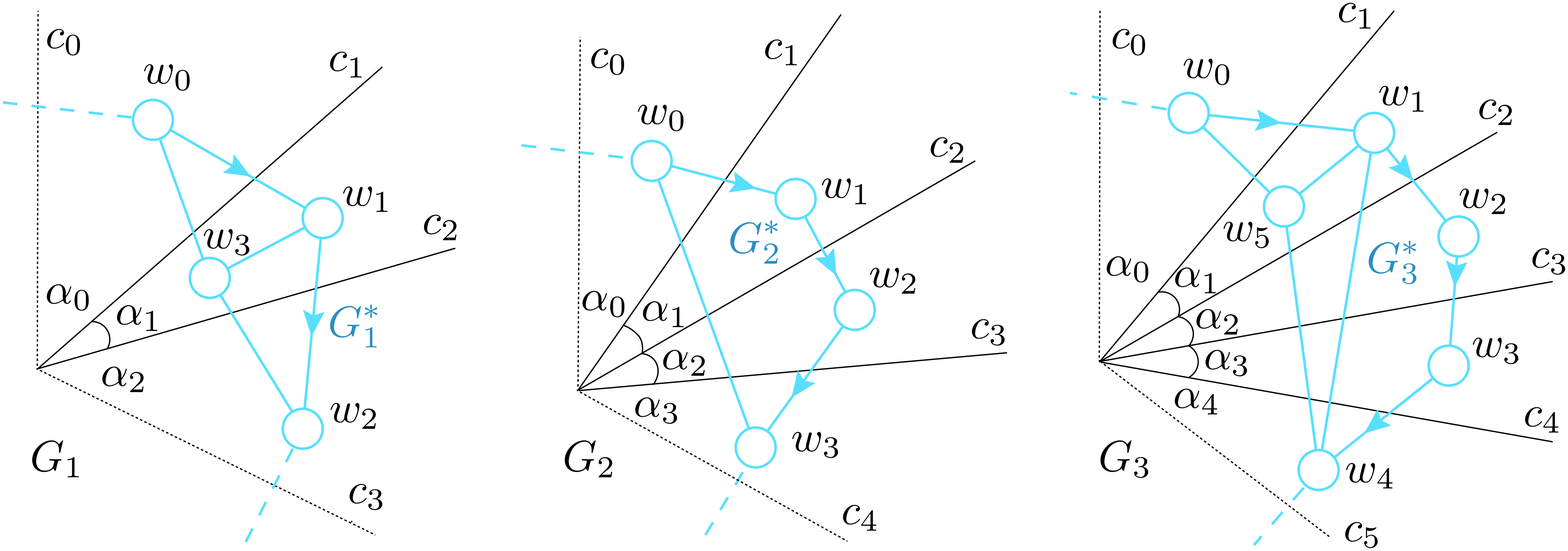}
    \caption{The baby SAW graphs for a local minimum of one, two, or three equal angles.}
    \label{fig:babes}
\end{figure}

\begin{proof} 
Using the notation shown in Figure~\ref{fig:babes}, we consider cases based on $k$.

{\bf Case 1: $k=1$.}  
The proof of this is the same as in the BLB degree-4 case:  We can note that if $s\in S(G_1^*)$ with $s(w_0)=0$ and $\mu$ is a MV assignment based on $s$ determined by the directed edges in $G_1^*$, then $s(w_2)=s(w_0)$, $s(w_1)=s(w_1)+\mu(c_1)$, and $s(w_2) = s(w_0) + \mu(c_1) + \mu(c_2)$, which implies that $\mu(c_1) \neq \mu(c_2)$. This gives us exactly the valid MV assignments of $G_1$ according to the Big-Little-Big Theorem, establishing the bijection between $M(G_1)$ and $S(G_1^*)$.


{\bf Case 2: $k=2$.}  
This proof is equivalent to the bird's foot degree-4 vertex:  Letting $s(w_0)=0$, we have $s(w_3)=s(w_0) +\mu(c_1)+\mu(c_2) +\mu(c_3)$. Since $s(w_0) \neq s(w_3)$, $\mu(c_1)+\mu(c_2)+\mu(c_3)=\pm 1$. This ensures the validity of MV assignments  $\mu$ on $G_2$ given by the directed edges and a proper 3-coloring of $G_2^*$ by the Big-Little-Big Theorem.


{\bf Case 3: $k=3$.} 

Here we have three equal angles in a row, and the Big-Little-Big Theorem says that any MV assignment $\mu$ will be valid among these creases if and only if $\sum_{i=1}^4 \mu(c_i)=0$, meaning $|M(G_3)|=\binom{4}{2}=6$.  The graph $G_3^*$ in Figure~\ref{fig:babes} has, if we pre-color $s(w_0)=0$, two choices for 3-coloring $w_1$, which then forces the colors of $w_5$ and $w_4$.  Then if $s(w_2)=s(w_4)$ we have 2 color choices for $w_3$, and if $s(w_2)\not=s(w_4)$ then there is only 1 choice for $w_3$.  Thus $S(G_3^*)=2(2+1)=6$.  Then, any coloring $s\in S(G_3^*)$ has $s(w_0)=s(w_4)$ and by the directed edges, $s(w_4)=s(w_0) +\mu(c_1)+\mu(c_2) +\mu(c_3)+\mu(c_4)$, which implies $\sum_{i=1}^4 \mu(c_i)=0$.  Thus the MV assignment generated by a coloring $s\in S(G_3^*)$ and the directed edges will be valid, and thus our map $f:M(G_3)\to S(G_3^*)$ is a bijection.


\end{proof}

Using Lemma~\ref{lem:babes}, we can construct SAW graphs for any flat-foldable vertex that is 3-nice.

\begin{theorem}\label{thm:3-nice}
For any 3-nice, single-vertex, flat-foldable crease pattern $C$, there exists a SAW graph $C^{*}$.
\end{theorem}

\begin{proof}
Let $C = C_0$ be an arbitrary 3-nice vertex of degree $2n$.  We proceed by induction on $n$.


For the base case, a flat-foldable vertex of degree 2 is just two creases that make sector angles of $180^\circ$ with each other.  (Such a vertex makes a straight line crease and can be thought of as degenerate, but it still satisfies Kawasaki's and Maekawa's Theorems.)  Its SAW graph is just two vertices $\{v_0, v_1\}$ connected by an edge, as in Figure~\ref{fig:baseCaseEx} left.  If $v_0$ is pre-colored, then the directed edge $(v_0,v_1)$ will act to form the bijection with the two MV assignments of the vertex.

Figure~\ref{fig:baseCaseEx} also illustrates  the basic idea of how the induction proceeds:  A local minimum of consecutive equal sector angles of $C_0$ are located and folded to make a smaller crease pattern $C_1$, which by induction has a SAW graph $C_1^*$.  We then use graph operations to modify $C_1^*$ into a SAW graph for $C_0$.  In Figure~\ref{fig:baseCaseEx} we see how the BLB degree-4  vertex SAW graph is a modification of the base case SAW graph, where a vertex and edge have been split and the baby SAW graph $G_1^*$ inserted.  Similarly, inserting the baby SAW graph $G_3^*$ gives a degree-6 vertex.  Also in Figure~\ref{fig:baseCaseEx} we see how the bird's foot SAW graph is a different modification of the base case.

\begin{figure}
    \centering
    \includegraphics[width=\linewidth]{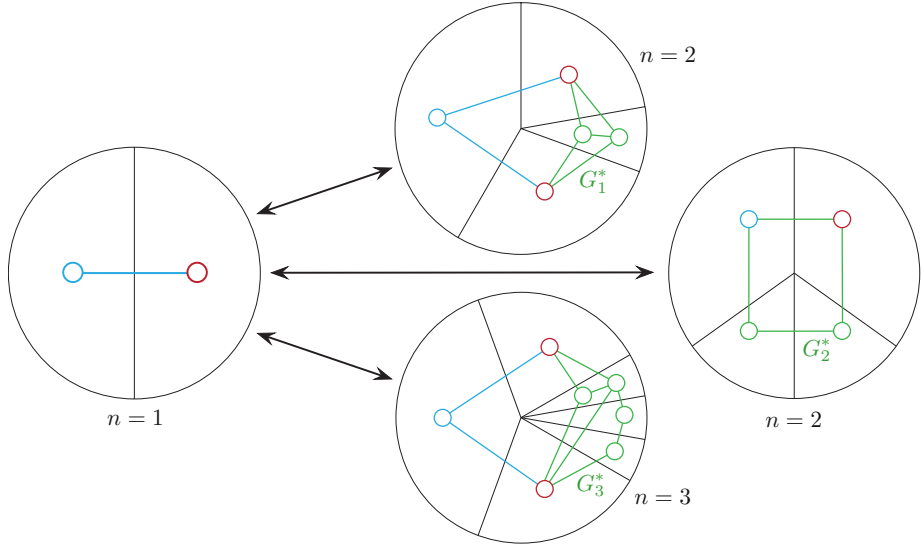}
    \caption{Example of the base case where $n = 1$, with recursive definitions of the next SAW graphs where $k = 1, 2$, and 3.}
    \label{fig:baseCaseEx}
\end{figure}

To make this more formal, let $c_i,\ldots, c_{i+k}$ be a sequence of creases in $C_0$ that have a local minimum of consecutive equal sector angles $\alpha_i=\ldots=\alpha_{i+k-1}$ between them (i.e., $\alpha_{i-1}>\alpha_i$ and $\alpha_{i+k}>\alpha_i$).  Since $C_0$ is 3-nice, we have that $k=1$, 2, or 3.  

For the cases where $k=1$ and $k=3$, when we fold the creases $c_i,\ldots, c_{i+k}$ we get a new crease pattern $C_1$ with the angles $\alpha_{i-1}, \ldots, \alpha_{i+k}$ replaced by a single angle $\alpha_i'=\alpha_{i-1}-\alpha_i+\alpha_{i+k}$.  By induction, we know a SAW graph  $C_1^*$ exists for $C_1$, and by the properties of SAW graphs, we know that we can embed $C_1^*$ onto the crease pattern of $C_1$ so that at least one vertex $v$ of $C_1^*$ is in the sector angle $\alpha_i'$.  To make the SAW graph $C_0^*$ we split the vertex $v$ into two vertices $v_1$ and $v_2$ so that every edge adjacent to $v$ is also adjacent to $v_1$ and $v_2$.  Then, if $k=1$ we insert the baby SAW graph $G_1^*$, identifying $w_0=v_1$ and $w_2=v_2$.  If $k=3$ then we insert $G_3^*$, identifying $w_0=v_1$ and $w_4=v_2$.  

That is, $V(C_0^*)=(V(C_1^*) \setminus\{v\})\cup V(G_k^*)$ and $E(C_0^*)=(E(C_1^*)\setminus {\rm adj}(v))\cup E(G_k^*) \cup N(v,w_0)\cup N(v,w_{2+k})$ where adj$(v)$ is the set of edges adjacent to $v$ and $N(v,x)$ is the set of edges with one endpoint in adj$(v)$ and the other endpoint $x$.

Then given any proper 3-coloring $s\in S(C_0^*)$, we have $s(w_0)=s(w_{k+1})$.  Letting $s(v)=s(w_0)$ induces a coloring $s_1\in S(C_1^*)$, and since $C_1^*$ is a SAW graph, we have a bijection $f_1:S(C_1^*)\to M(C_1)$.  The coloring $s$ also induces a proper 3-coloring $s_2\in S(G_k^*)$, where we assume the vertex $w_0$ is pre-colored, and we have a bijection $f_2:S(G_k^*)\to M(G_k)$.  Since $G_k$ shares no creases with $C_1$ and valid MV assignments of $G_k$ operate independently of those of $C_1$ (as implied by the Big-Little-Big Theorem), we have that $f_1$ and $f_2$ form a bijection between $S(C_0^*)$ and $M(C_0)$, where $|S(C_0^*)|=|S(C_1^*)|\cdot |S(G_k^*)|$ and $|M(C_0)|=|M(C_1)|\cdot |M(G_k)|$.

When $k=2$ we have creases $c_i, c_{i+1}$, and $c_{i+2}$ with $\alpha_i=\alpha_{i+1}$ between them and $\alpha_{i-1}>\alpha_i$ and $\alpha_{i+2}>\alpha_i$.  Folding these gives a single vertex crease pattern $C_1$ with creases $c_{i+1}, c_{i+2}$ and angles $\alpha_i, \alpha_{i+1}$ removed, i.e., $C_1$ will have the sequence of creases $c_{i-1}, c_i$, and $c_{i+3}$ with angles $\alpha_{i-1}$ and $\alpha_{i+2}$ between them.  Then the SAW graph $C_1^*$ exists, and by SAW graph properties there will exist an edge $\{v_1,v_2\}\in E(C_1^*)$ crossing the crease $c_i$, and let us assume that for the bijection $f_1:S(C_1^*)\to M(C_1)$ we have the directed edge $(v_1,v_2)$.  We relabel the vertices $v_1, v_2$ with $w_0, w_3$, respectively, and define the graph $C_0^*$ by $V(C_0^*)=V(C_1^*)\cup\{w_1, w_2\}$ and $E(C_0^*)=E(C_1^*)\cup \{w_0,w_1\}\cup \{w_1,w_2\}\cup\{w_2,w_3\}$, which is simply adding the baby SAW graph $G_2^*$ onto the edge $\{w_0,w_3\}$ of $C_1^*$.  

Now let $s\in S(C_0^*)$.  Since $C_1^*$ and $G_2^*$ are subgraphs of $C_0^*$, we may consider $s\in S(C_1^*)$ and $s\in S(G_2^*)$ as well (where we think of $w_0$ in $G_2^*$ to be pre-colored).  Let $f_2:S(G_2^*)\to M(G_2)$ be the bijection for the SAW graph $G_2^*$ and the creases $c_i, c_{i+1}, c_{i+2}$ that form the copy of $G_2$ in $C_0$, and let $\mu_1=f_1(s)$ and $\mu_2=f_2(s)$.  Then $\mu_1(c_i)=s(w_3)-s(w_0) = s(w_3)-s(w_2)+s(w_2)-s(w_1)+s(w_1)-s(w_0) = \mu_2(c_{i+2})+\mu_2(c_{i+1})+\mu_2(c_i)$ (see the directed edges of $G_2^*$ in Figure~\ref{fig:babes} for reference).  This means that in $C_1$, $\mu_1(c_i)$ equals the majority MV assignment under $\mu_2$ among the creases $c_i, c_{i+1}, c_{i+2}$ in $C_0$.  Therefore the MV assignment $\mu$ for $C_0$ defined by
$$\mu(c)=\left\{
\begin{array}{cl}
\mu_1(c) & \mbox{if }c\in E(C_1)\setminus\{c_i\}\\
\mu_2(c) & \mbox{if }c\in \{c_i, c_{i+1}, c_{i+2}\}
\end{array}\right.$$
is valid (since $\mu_1$ is valid on $C_1$, $\mu_2$ is valid on $\{c_i,c_{i+1},c_{i+2}\}$, and $\mu_1(c_i) = \mu_2(c_{i+2})+\mu_2(c_{i+1})+\mu_2(c_i)$ means that Maekawa's Theorem will hold for $\mu$).  Since $\mu_1$ and $\mu_2$ are bijections, we conclude that this mapping from 3-colorings $s\in S(C_0)$ to MV assignments $\mu\in M(C_0)$ is also a bijection.

\end{proof}



Theorem~\ref{thm:deg4} covers only a portion of all single-vertex crease patterns, those that are 3-nice.  In order to prove this for all flat-foldable vertices, baby gadgets $G_k^*$ are needed for an arbitrary number of consecutive equal angles as well as SAW graphs for the all-equal-angles cases of degree six and higher.  Nonetheless, we now have SAW graphs for a wide variety of flat-foldable vertices, including all degree-4 vertices (since we already have a SAW graph for the all-equal-angles degree-4 case).

\begin{figure}
    \centering
    \includegraphics[width=\linewidth]{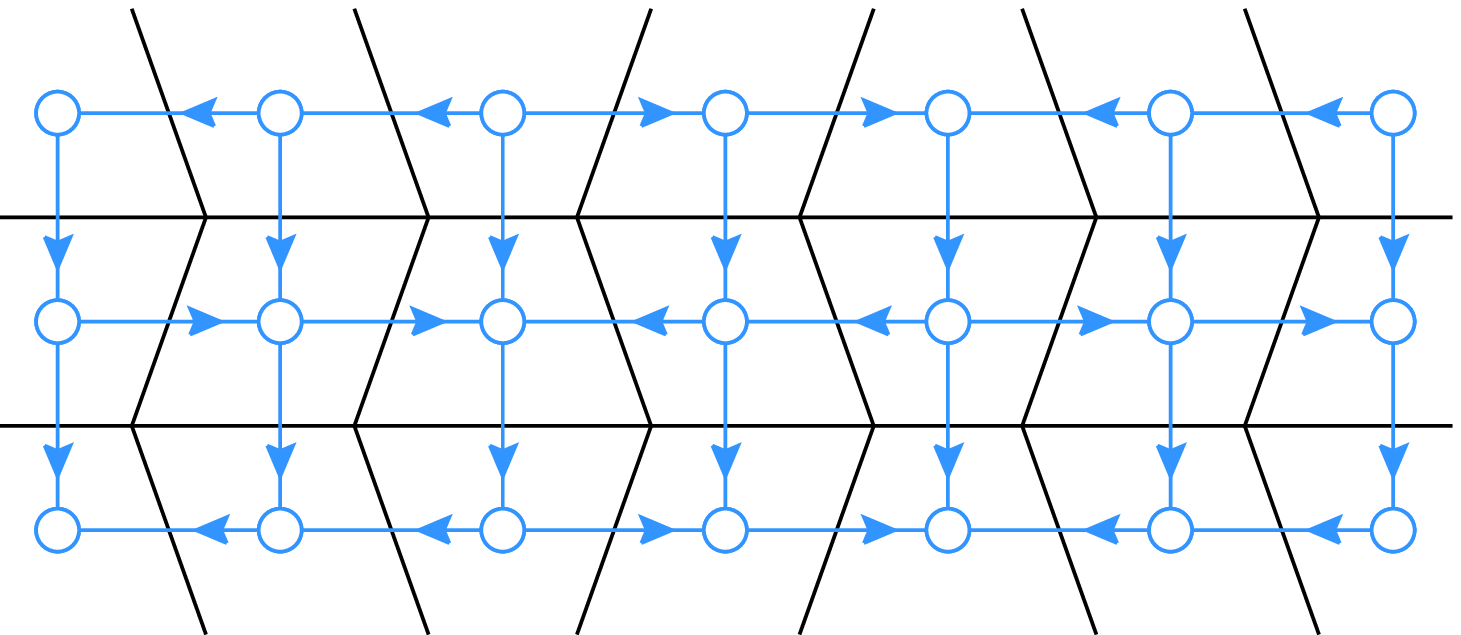}
    \caption{Modified Miura-ori crease patterns have the same (ignoring edge orientations) SAW graph as the standard Miura-ori.}
    \label{fig:modMiura}
\end{figure}

\section{Tiling SAW Graphs}\label{sec:tiling}

Now that we have seen SAW graphs for a variety of single-vertex crease patterns, we turn to the problem of putting them together to construct SAW graphs of multiple-vertex crease patterns.  
In a general, locally flat-foldable crease pattern  $(C,P)$, a MV assignment will be locally valid if each vertex is valid. However, we have that every crease $c$ that does not terminate on the boundary of $P$ is shared by two vertices which forces $c$ to fold in a way that results in both vertices being valid. Since the directed edges of a single-vertex SAW graph are boundary edges (in the sense described in Section~\ref{sec:backgrounnd}), we may try to maintain MV consistency between adjacent vertices in the crease pattern by tiling single-vertex SAW graphs, i.e., identify two boundary edges of two single-vertex SAW graphs that cross the same crease line $c$ and are oriented in the same direction.  This is not always possible to do when tiling many different SAW graphs, but in some cases this simple strategy works very well.


For example, a \emph{modified Miura-ori} crease pattern is shown in Figure~\ref{fig:modMiura}, which is the same as a Miura-ori crease pattern (Figure~\ref{fig:Miura}(a)) with some of the vertical columns of zig-zag creases reflected from left-to-right. Notice that even with some of the columns of crease pattern vertices (and thus their bird's feet SAW graphs) reflected from left-to-right, the directed edges between neighboring SAW graphs will match. This implies that, ignoring the edge orientations, the SAW graph for an $m\times n$ modified Miura-ori is the same $m\times n$ grid graph as for a $m\times n$ Miura-ori. We arrive at the following rather surprising result:

\begin{theorem}
The number of locally-valid MV assignments for any $m\times n$ modified Miura-ori is equal to those of an $m\times n$ standard Miura-ori.
\end{theorem}

Similarly, a \emph{snake tessellation} crease pattern appears similar to the modified Miura-Ori but with degree-6 vertices (called waterbomb vertices) as well; see Figure~\ref{fig:snake}(b). However, if we imagine splitting each of the degree-6 vertices into two bird's feet, as shown in Figure~\ref{fig:snake}(a), we see that the number of valid MV assignments for a waterbomb vertex is the same as for two bird's feet that share a heel.  That is, by the Big-Little-Big Theorem, in order for a waterbomb vertex to fold flat, the creases $\{w_1, w_2, w_3\}$ in Figure~\ref{fig:snake}(a) cannot all have the same MV parity--say, exactly one is a valley--and then exactly one of $\{w_4, w_5, w_6\}$ must be a valley as well (to satisfy Maekawa's Theorem).  If we insert a bird's heel $c$ in between these creases, we have that these two bird's feet will fold flat under these same exact conditions for the creases $\{w_1, w_2, w_3\}$ and $\{w_4, w_5, w_6\}$.  Therefore, the number of locally-valid MV assignments for the snake tessellation is the same as for the corresponding modified Miura-ori, such as that shown in Figure~\ref{fig:snake}(b).  We have thus proved the following:

\begin{figure}
    \centering
    \includegraphics[width=\linewidth]{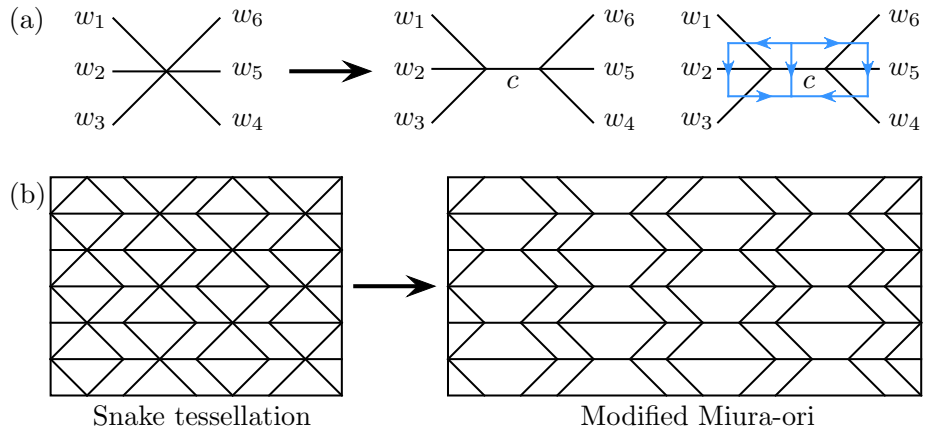}
    \caption{(a) Splitting a waterbomb vertex into two bird's feet.  (b) The snake tessellation turns into a modified Miura-ori under this transformation.}
    \label{fig:snake}
\end{figure}

\begin{theorem}
The number of locally-valid MV assignments for any $m\times n$ snake tessellation is equal to those of an $m\times n$ standard Miura-ori or $m\times n$ modified Miura-ori.
\end{theorem}

These examples show us how sometimes it is easy to tile single-vertex SAW graphs to make SAW graphs for more general crease patterns. However, the condition that boundary edges crossing the same crease have to be oriented in the same direction is not always achievable. 

\begin{figure}
    \centering
    \includegraphics[width=\linewidth]{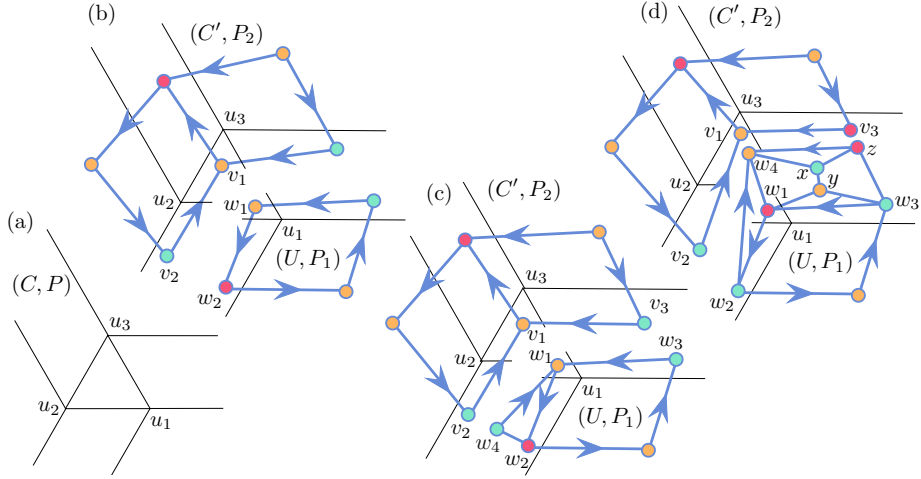}
    \caption{(a) A triangle twist.  (b) The twist split into two crease patterns, each with a SAW graph that cannot be merged together.  (c) A triangle added to one SAW graph to allow merging.  (d) A triangle and a triangular prism added to one SAW graph to allow merging.}\label{fig:tritwist}
    \label{fig:twist}
\end{figure}

For example, consider the triangle twist crease pattern $(C,P)$ in Figure~\ref{fig:tritwist}(a), in which each vertex is a bird's foot. To find its SAW graph, we first find the SAW graph for two of the  crease pattern vertices. Let the three vertices of $C$ be $u_1, u_2, u_3$. We divide the region $P$ into two disjoint regions $P_1$ and $P_2$, where $u_1 \in P_1, u_2, u_3 \in P_2$, and $P = P_1 \cup P_2$. Now let $(U,P_1)$ be the single-vertex crease pattern corresponding to $u_1$, and $(C',P_2)$ be the crease pattern corresponding to the two vertices $u_2$ and $u_3$, where $C' = C \setminus{\{u\}}$. We can find a SAW graph for crease pattern $(C', P_2)$ by simply connecting the SAW graphs for the two vertices along one boundary edge; see Figure~\ref{fig:tritwist}(b). However, in order to add in the third vertex to get the SAW graph for $(C, P)$, we have to match the SAW graphs $C'^{*}$ and $U^{*}$ along two boundary edges. This is because crease patterns $(C',P_2)$ and $(U, P_1)$ share two creases, so the two boundary edges along those two creases must also have the same orientation. 

Notice that we cannot just merge the vertices because the two edges $(v_2, v_1) \in E(C'^*)$ and $(w_1,w_2)\in E(U^*)$ cross the same crease but do not have the same orientation. In order to switch its orientation, we have to modify our SAW graph $U^{*}$. To do this, we add a triangle on the edge $(w_1, w_2)$. That is, we add a new vertex $w_4$, and add edges $\{w_4, w_1\}, \{w_4, w_2\}$. Since this forms a $3$-cycle, the color for $w_4$ is forced. Now we can choose one of the two new edges to be directed in the direction opposite to $(w_1,w_2)$, so as to match with $(v_2,v_1)$. However, this adds an extra undirected edge to the boundary of $U^*$. If we choose to have the directed edge $(w_4, w_1)$, then we can merge the two SAW graphs by setting $v_2 = w_4, v_1 = w_1, v_3 = w_3$; see Figure~\ref{fig:tritwist}(c). 

If we choose to have the directed edge $(w_4, w_2)$, the undirected edge will be between the two directed edges of $U^*$ that we want to merge with $C'^*$. To deal with this scenario, we add a triangular prism graph (with new vertices $x, y$, and $z$) to $U^*$ along $\{w_1,w_3\}$, as seen in Figure~\ref{fig:tritwist}(d). This swaps the positioning of a directed and undirected edge, allowing us to merge the SAW graphs by identifying $v_2=w_2$, $v_1=w_4$, and $v_3=z$. Note that when we add these new graphs, the directed edges $(w_1, w_2)$ and $(w_3, w_1)$ are no longer boundary edges! However, the merged SAW graphs will create a SAW graph for the original crease pattern, as we will now prove in the following Lemmas.

\begin{lemma}\label{lem:sawtriangle}
Let $C$ be a crease pattern with a SAW graph $C^{*}$. Let edge $e = (u, v) \in E(C^{*})$ be a boundary edge crossing crease $c$.  Create a new graph $C^{**}$ 
by $V(C^{**})=V(C^*)\cup\{w\}$ and $E(C^{**})=E(C^*)\cup\{ (w,u), \{v,w\} \}$.
(see Figure~\ref{fig:sawadd}(b)). Then  $C^{**}$ is also a SAW graph for crease pattern $C$.
\end{lemma}

\begin{proof} 
We are given that $C^*$ is a SAW graph of $C$, so there exists a bijection $f:M(C)\to S(C^*)$.  Note that since $\{u,v,w\}$ is a triangle in $C^{**}$, any 3-coloring $s\in S(C^*)$ will, if applied to $C^{**}$, force the color of $w$, which we will denote by $s(w)$ and consider $s$ to be a 3-coloring of $C^{**}$ as well.   Now, the triangle $\{u,v,w\}$ gives us
\begin{align}\label{eq:triaddSAW}
s(u) + s(v) + s(w) \equiv 0 \pmod{3}.
\end{align}
We next examine how the directed edges $(u,v)$ and $(w,u)$ determine the value of $\mu(c)$:
\begin{align*}
    \mu(c) &\equiv s(u) - s(w) \pmod{3} \\
    &\equiv s(u) +s(u) + s(v) \pmod{3}\ (\mbox{substituting }\eqref{eq:triaddSAW})\\
    &\equiv s(v) + 2s(u)  \pmod{3}\\
    &\equiv s(v) - s(u)  \pmod{3}.
\end{align*}
Therefore the directed edges $(u,v)$ and $(w,u)$ are consistent in how they determine the MV assignment of crease $c$.  Thus $C^{**}$ does not change the number of 3-colorings of $C^*$ nor changes how these colorings biject with valid MV assignments of $C$.  This implies that $C^{**}$ is also a SAW graph of $C$.

\end{proof}

\begin{figure}
    \centering
    \includegraphics[scale=.5]{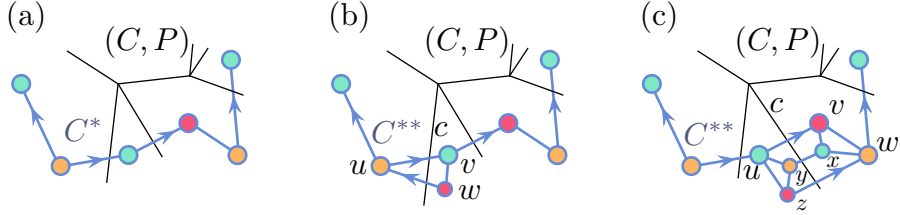}
    \caption{(a) A part of a crease pattern $(C,P)$ and the boundary edges of its SAW graph $C^*$.  (b) Adding a triangle to $C^*$. (c) Adding a triangular prism to $C^*$.}
    \label{fig:sawadd}
\end{figure}

Notice that creating $C^{**}$ according to Lemma~\ref{lem:sawtriangle} gives us a SAW graph that is functionally identical to $C^*$ but whose outer face has been modified.  By adding the triangle, a directed edge has been flipped and an undirected edge has been added next to it on the outer face. Being able to reverse directed edges will help us tile SAW graphs, but we need a way to manage any undirected edges on the boundary that are introduced by such an operation.  The next Lemma provides such a way.

\begin{lemma}\label{lem:triangularprism}
Let $C$ be a crease pattern with a SAW graph $C^{*}$. Let edge $e_1 = (u, v) \in E(C^{*})$ be a boundary edge crossing crease $c$, and edge $e_2 = \{v, w\} \in E(C^{*})$ be an undirected boundary edge adjacent to $e_1$. Then we may attach a triangular prism graph to edges $e_1$ and $e_2$ (as shown in Figure~\ref{fig:sawadd}(c)) to create a graph $C^{**}$ that is also a SAW graph for crease pattern $C$.
\end{lemma}

\begin{proof}
Since $C^{*}$ is a SAW graph of crease pattern $C$, there exists a bijection $f: M(C)\to S(C^{*})$. Let $s \in S(C^{*})$ be an arbitrary proper $3$-coloring of the original SAW graph. To prove that $C^{**}$ is a SAW graph on $C$, we must show that any coloring on $C^{*}$ forces a unique coloring on $C^{**}$, and that this new coloring bijects to the same locally-valid MV assignment $f(s)$ of $C$.
 
Note that both $\{u,y,z\}$ and $\{v,x,w\}$ are triangles. Thus, the colorings on these vertices need to satisfy the following equations:
\begin{align}
    s(u) + s(y) + s(z) \equiv 0 \pmod{3} \\
    s(v) + s(x) + s(w) \equiv 0 \pmod{3}
\end{align}
Subtracting equation $(2)$ from $(3)$ yields the following:
\begin{align}\label{eq:1}
    (s(v) - s(u)) + (s(x) - s(y)) + (s(w) - s(z)) \equiv 0 \pmod{3}
\end{align}
Since we have edges $\{u,v\}, \{y,x\},$ and $\{z,w\}$ in our graph $C^{**}$, none of the three terms in parentheses on the left-hand side of Equation~\eqref{eq:1} can be zero. In fact, we have
\begin{align*}
    0 &\equiv (s(v) - s(u)) + (s(x) - s(y)) + (s(w) - s(z)) \pmod{3} \\
    &\equiv \pm 1 \pm 1 \pm 1
\end{align*}
The only way for this to sum to $0$ mod 3 is if each term in parentheses is $1$ or each is $-1$, thus:
$$ s(v) - s(u) \equiv s(x) - s(y) \equiv s(w) - s(z) \pmod{3}.$$

In fact, the values of $s(u), s(v),$ and $s(w)$ have already been determined by the coloring $s$ on $C^{*}$, so the value of these three differences are forced. We can now solve for the values of the other colorings in terms of the colorings of vertices $u, v, w$:
\begin{align*}
s(x) & \equiv -s(v) - s(w) \\
s(y) & \equiv s(x) + s(u) - s(v) \equiv -2s(v) - s(w) + s(u) \\
s(z) & \equiv s(w) + s(u) - s(v)
\end{align*}
Thus, a coloring on $C^{*}$ forces a coloring of the new vertices in $C^{**}$. We also have that the directed edges $(u,v)$ and $(z,w)$ are consistent in determining the value of $\mu(c)$ since $\mu(c) \equiv s(v) - s(u) \equiv s(w) - s(z)$. So we can now conclude that $C^{**}$ is a SAW graph on crease pattern $C$.

\end{proof}

\begin{figure}
    \centering
    \includegraphics[width=\linewidth]{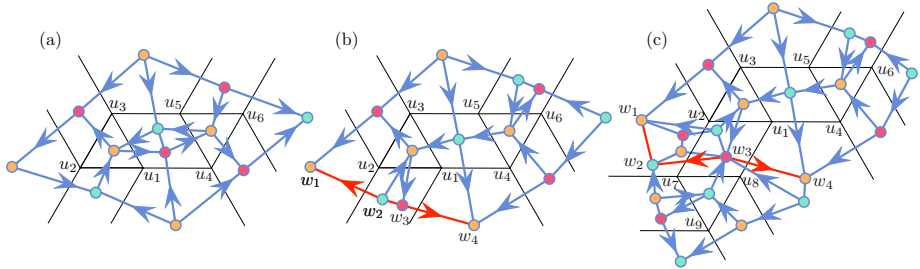}
    \caption{(a) A tiling of SAW graphs for the triangle twist that does not work.  (b) A tiling that does work. (c) Extending this tiling to three triangle twists.}
    \label{fig:tilingex}
\end{figure}

Once again, our goal is to modify the orientation of the boundary edges of any given SAW graph. By adding a triangular prism graph, we swap the places of two adjacent edges, where one is directed and the other is undirected.  This will help us tile SAW graphs.

In fact, when tiling SAW graphs we only want to merge directed boundary edges, not undirected edges. For example, we saw in Figure~\ref{fig:tritwist}(c) how we can create a SAW graph for a triangle twist by inserting a triangle between two of the bird's foot SAW graphs.  That places an undirected edge on this SAW graph for the triangle twist.  Suppose our crease pattern was two joined triangle twists, and we attempted to make a SAW graph for this by merging two directed and the one undirected edge of two triangle twist SAW graphs, as shown in Figure~\ref{fig:tilingex}(a).  This does not work; there are 170 locally-valid MV assignments of the crease pattern, but only 110 proper 3-colorings of this graph (with one vertex pre-colored).  The reason for this discrepancy is that merged undirected  edges made two triangles in the SAW graph that share an edge, and the coloring of these triangles, together with their directed edges, force the creases $\{u_1, u_3\}$ and $\{u_4, u_5\}$ to have the same MV parity, which is not necessary in a locally-valid MV assignment for this crease pattern.

The tiling shown in Figure~\ref{fig:tilingex}(b) does work, since the undirected edges, and thus their triangles, are kept on the boundary of the SAW graph.  But suppose we wanted to add another triangle twist to this crease pattern, as in Figure~\ref{fig:tilingex}(c)?  Then we would need to alter the boundary path $\{w_1, w_2, w_3, w_4\}$ of the SAW graph in Figure~\ref{fig:tilingex}(b), shown in red, to move the undirected edge to an endpoint of this path.  This is done by inserting a triangular prism graph, shown in Figure~\ref{fig:tilingex}(c).  Then a SAW graph for the triangle twist may be tiled onto this red path without interfering with the undirected edges (while making sure that the directed edges around the bird's feet creases are as they should be).

In this way, Lemmas~\ref{lem:sawtriangle} and \ref{lem:triangularprism} allow us to merge, or tile, two SAW graphs along boundary edges no matter what the configuration of directed edges are. We now show that this merging process preserves the bijection within both of the crease patterns. This theorem allows us to match together SAW graphs for single-vertex crease patterns to create SAW graphs for general crease patterns.

\begin{theorem}\label{thm:tiling}
Let $C$ be a crease pattern such that for all vertices $v \in V(C)$, there exists a SAW graph $V^{*}$ where $V$ is the single-vertex crease pattern corresponding to $v$. Then, there exists a SAW graph for the whole crease pattern $C$.
\end{theorem}

\begin{proof}
We proceed by induction on the  number of vertices in a crease pattern.

Let $(C,P)$ be a locally flat-foldable crease pattern on a bounded region $P\subset \mathbb{R}^2$ such that for every vertex $v$ in $C$, there exists a SAW graph $V^{*}$ where $V$ is the single-vertex crease pattern corresponding to vertex $v$. Let $u$ be a vertex of $C$ such that at least one of the creases adjacent to $u$ terminates on the boundary of $P$.  Let $\gamma$ be a simple curve starting and ending on different boundary points of $P$ and that crosses only the creases adjacent to $u$ that do not terminate on the boundary of $P$ and crosses no other elements of $C$. Let $\overline{P}=P\setminus \rm{int}(\gamma)$ and let $(\overline{C}, \overline{P})$ be the crease pattern made from $C$ by ``clipping" away the vertex $u$ (where any crease lines adjacent to $u$ and another vertex in $C$ are now creases that terminate on the boundary of $\overline{P}$); see Figure~\ref{fig:tileproof}(a).

\begin{figure}
    \centering
    \includegraphics[width=\linewidth]{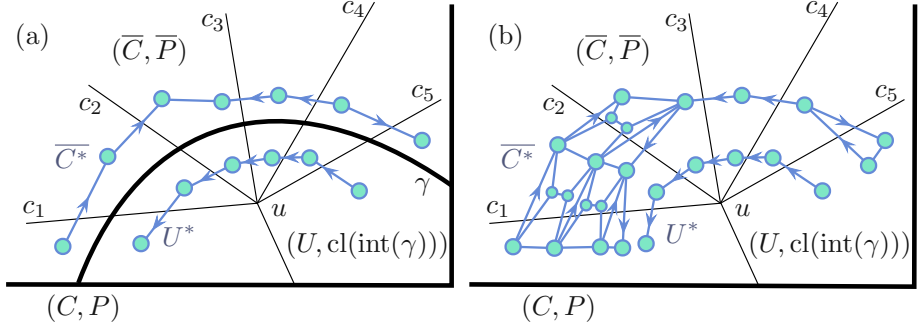}
    \caption{(a) Splitting a crease pattern $(C,P)$ into $(\overline{C},\overline{P})$ and $(U,\mbox{cl(int(}\gamma)))$. (b) The triangular prisms and triangles needed to add to $\overline{C^*}$ to make its boundary match that of $U^*$.}
    \label{fig:tileproof}
\end{figure}

Our graph $\overline{C}$ has one less vertex, thus we can say that it has a valid SAW graph $\overline{C^{*}}$ by our inductive hypothesis. So there exists a bijection $f': S(\overline{C^{*}}) \to M(\overline{C})$. For our single vertex $u$, let $(U, \textrm{cl(int}(\gamma)))$ be its crease pattern (where cl$(A)$ is the closure of the set $A$). We also know that it has a valid SAW graph $U^{*}$. 
We call the bijection $f_u: S(U^{*}) \to M(U)$. 

We now proceed by constructing the graph $C^{*}$ using both SAW graphs $\overline{C^{*}}$ and $U^{*}$. Let $c_1, c_2, \dotsc, c_k$ be the creases shared by the two crease patterns $\overline{C}$ and $U$. Let $B$ be the union of all faces in our crease pattern $C$ crossed by the simple curve $\gamma$. This is the region shared between the two crease patterns $\overline{C}$ and $U$, which is what will focus on. We want to merge all boundary edges in this region $B$ for both SAW graphs $\overline{C^{*}}$ and $U^{*}$. To preserve the bijection in both crease patterns, we need that the boundary along this region of both SAW graphs has no undirected edges and two edges crossing the same crease have the same orientation. To do this, we use Lemmas~\ref{lem:sawtriangle} and \ref{lem:triangularprism} to modify the boundary edges of $\overline{C^*}$ that lie in $B$  by first inserting triangles to reverse any directed edges crossing a crease (say $c_i$) that do not have the same orientation as the corresponding directed edge of $U^*$ (that also crosses $c_i$).  Then any undirected edges of this boundary of $\overline{C^*}$ are moved to the periphery of $B$ by inserting triangular prisms.  Finally, any undirected edges of $U^*$ that lie in $B$ are also moved to the periphery by inserting triangular prisms.  

Then we may let $C^*$ be the graph obtained by merging our modified $\overline{C^*}$ and $U^*$ graphs along the region $B$ created by our curve $\gamma$.  As seen in Lemmas~\ref{lem:sawtriangle} and \ref{lem:triangularprism}, these modified versions of $\overline{C^*}$ and $U^*$ also serve as SAW graphs for $\overline{C}$ and $U$, respectively, and thus use the same functions $f'$ and $f_u$ to biject from colorings to MV assignments.  We claim that we can create a new function $f:S(C^*)\to M(C)$ by combining $f'$ and $f_u$ as follows:  For any $s\in S(C^*)$, let $s'$ be the coloring this creates for $\overline{C^*}$ and $s_u$ be the coloring this places on the vertices in $U^*$.  Then we let $f(s)$ be the MV assignment on $C$ that uses $f'(s')$ for all creases $c\in \overline{C}$ and $f_u(s_u)$ for all creases in $U$.  Note that this definition for $f$ is consistent on all creases $c_i\in \overline{C}\cap U$ because we made sure that the modified SAW graphs for $\overline{C^*}$ and $U^*$ have directed edges with the same orientation crossing $c_i$.  Thus $f$ is well-defined and $f(s)\in  M(C)$ for any $s\in S(C^*)$.  

That $f:S(C^*)\to M(C)$ is a bijection follows from the fact that $f'$ and $f_u$ are bijections: If $\mu\in M(C)$, then it induces MV  assignments $\overline{\mu}\in M(\overline{C})$ and $\mu_u\in M(U)$.  Then  there exist $s_1\in S(\overline{C^*})$ with $f'(s_1)=\overline{\mu}$ and $s_2\in S(U^*)$ with  $f_u(s_2)=\mu_u$.  Because the merged boundary edges of $\overline{C^*}$ and $U^*$ have the  same orientation, the colorings $s_1$ and $s_2$ can be combined to give a coloring $s\in S(C^*)$ with $f(s)=\mu$, proving surjectivity of $f$.  If $f(s_1)=f(s_2)$ for some colorings $s_1,s_2\in S(C^*)$, then we can similarly split the MV assignments $f(s_1)$ and $f(s_2)$ to induce MV assignments on $\overline{C}$ and $U$ and use $f'$ and $f_u$ to show that the colorings $s_1$ and $s_2$ are the same, proving injectivity.

\end{proof}

A very popular class of flat-foldable origami crease patterns for engineering applications are those that contain only degree-4 vertices \cite{Evans}.  Theorem~\ref{thm:tiling} together with the degree-4 SAW graphs from Figure~\ref{fig:deg4} give us the following corollary.

\begin{corollary}
Every finite, flat-foldable crease pattern that is 4-regular has a SAW graph.
\end{corollary}

As a final example, we consider the classic crane model, also known as the flapping bird, shown in Figure~\ref{fig:crane}.  The SAW graph for this crease pattern is made by merging numerous SAW graphs for degree-4 vertices, although note that to do this we need to use alternate orientations of the arrows from those of Figure~\ref{fig:deg4}(c), as explained in Remark~\ref{remark1}.  Also, the center vertex of the crane's crease pattern is the waterbomb vertex, which we dealt with as we did in the snake tessellation (Figure~\ref{fig:snake}(a)).  Entering this SAW graph into {\em Mathematica} or Sage reveals that it has 93,313 ways to properly 3-color the vertices with one vertex pre-colored.

\begin{figure}
    \centering
    \includegraphics[scale=.5]{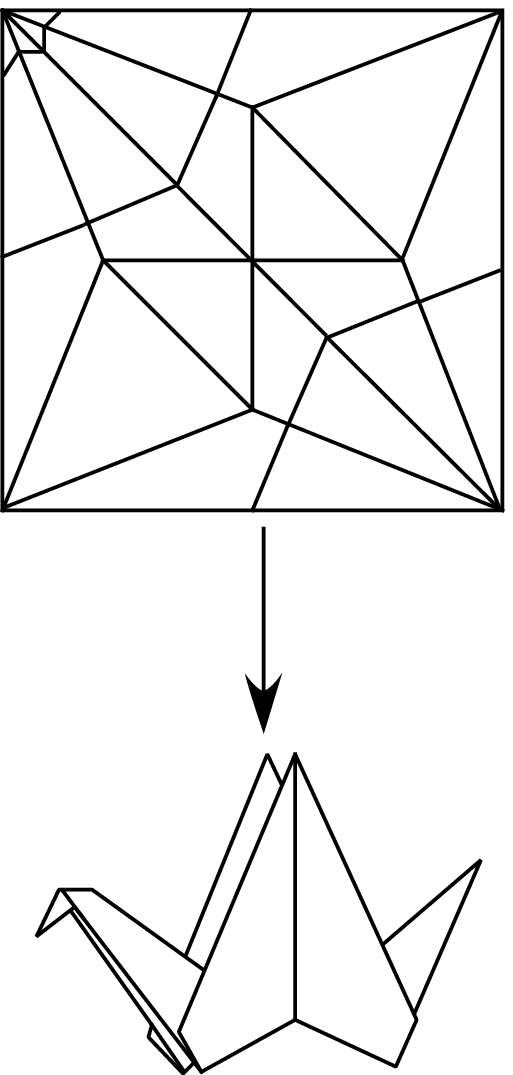} \hspace{.1in}
    \includegraphics[scale=.3]{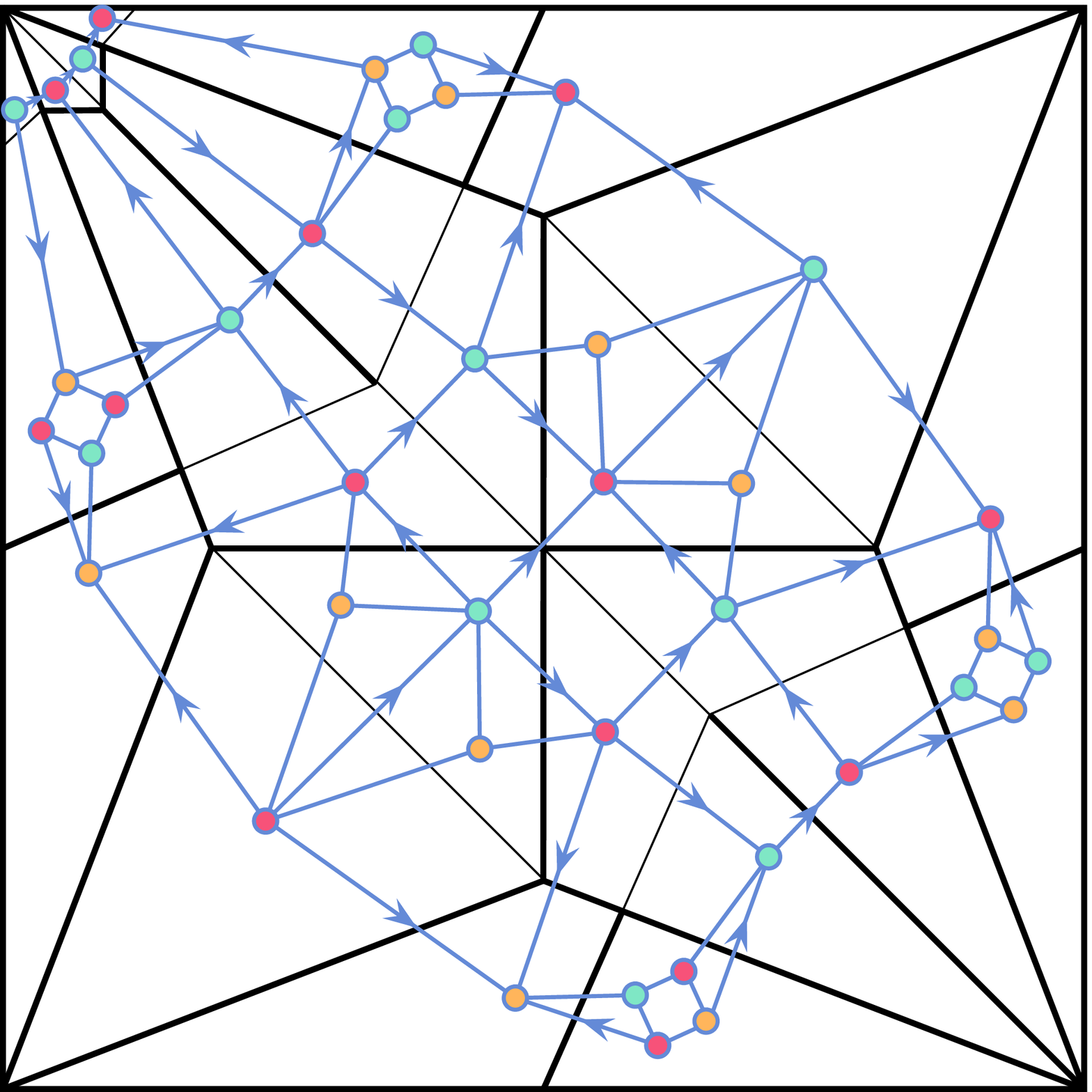}
    \caption{The classic crane (aka flapping bird) origami model, with its crease pattern and SAW graph. Letting the yellow, green, and red vertices be 0, 1, and 2 respectively gives the MV assignment shown.}
    \label{fig:crane}
\end{figure}

\section{Conclusion}

We have shown that a wide class of flat-foldable crease patterns have SAW graphs, and therefore their locally-valid MV assignments can be enumerated using graph coloring algorithms.  Note, however, that we do not yet have SAW graphs for flat-foldable vertices that (a) are strictly $m$-nice for $m>3$ and (b) have all equal sector angles and degree greater than 4.  While further work is required to expand these results beyond 3-nice vertices, this provides very strong evidence that the basic combinatorial structure underlying locally-valid MV assignments in flat origami is 3-colorings of graphs.  

Furthermore, the prior work of \cite{miura} that proved Miura-ori  foldings were equivalent to 3-coloring grid graphs inspired and gave a valuable tool for studying flat-foldable origami from a statistical mechanics perspective (see \cite{Assis}), creating exactly solvable models that show phase transitions exist in some origami tessellations, like the Miura-ori.  The larger family of 3-colorable graphs that we present in this paper could help expand such statistical mechanics results for wider assortments of crease patterns.

\section{Acknowledgments}

This material is based upon work supported by the National Science Foundation under Grant Number DMS-1851842, the MathILy-EST Research Experience for Undergraduates, under the direction in 2019 of the third author. Any opinions, findings, and conclusions or recommendations expressed in this material are those of the authors and do not necessarily reflect the views of the National Science Foundation.  The authors wish to thank Bryn Mawr College for providing excellent hospitality and facilities for conducting this work.


\end{document}